\documentclass[11pt,dvips,leqno]{article} 
                          
\usepackage{amsmath}
\usepackage{theorem}
\usepackage{amsfonts}
\usepackage{amsbsy}
\usepackage{amssymb}

\newtheorem{satz}{Theorem}[section]
\newtheorem{lemm}[satz]{Lemma}
\newtheorem{koro}[satz]{Corollary}
{\theorembodyfont{\upshape} \newtheorem{beme}[satz]{Remark}}

\newenvironment{proof}{{\it Proof:}}{\mbox{}\hfill$\Box$}

\newcommand{\img}{\operatorname{im}} 
\newcommand{\id}{\operatorname{id}} 
\newcommand{\ord}{\operatorname{ord}} 
\newcommand{\divi}{\operatorname{div}} 
\newcommand{\spec}{\operatorname{Spec}} 
\newcommand{\spp}{\operatorname{supp}} 
\newcommand{\gal}{\operatorname{Gal}} 
\newcommand{\sgn}{\operatorname{sgn}} 
\newcommand{\ext}{\operatorname{Ext}} 
 
\begin{document} 

{\bf\huge Mahler measures, K-theory and values of 
L-functions}
 \\
 \\
by Hubert Bornhorn \\ 
\setcounter{section}{-1}

{\bf Abstract} The Mahler measure of a polynomial $P$ in $n$ variables 
is defined as the mean of $\log|P|$ over the $n$-dimensional torus. 
For certain polynomials with integer coefficients in two variables the Mahler 
measure is known to be related to special values of L-functions of arithmetic 
objects (e.g. Dirichlet characters and elliptic curves over $\mathbb{Q}$). 
Inspired by work of Deninger (\cite{deni97}) Boyd has investigated this 
relationship numerically (\cite{boyd98}). In this paper we reduce some 
conjectures of Boyd to Beilinson`s conjectures on special values of 
L-functions. The methods in use are widely of K-theoretical nature.   

\section{Introduction}

The (logarithmic) Mahler measure of a polynomial 
$P\in\mathbb{C}[t_1,t_2]$ is defined as 
\begin{eqnarray*}
m(P)    & := &  \frac{1}{(2\pi i)^2}\int\limits_{T^2}\log |P(z_1,z_2)|
\frac{dz_1}{z_1}\wedge\frac{dz_2}{z_2} \\
        & = &   \int_0^1 \int_0^1\log |P(e^{2\pi i \alpha_1},
e^{2\pi i \alpha_2})| d\alpha_1 d\alpha_2                        
\end{eqnarray*}
where $T^2:=S^1\times S^1\subset\mathbb{C}^2$ is the real  
$2$-torus. \\
In \cite{smyt81b} Smyth discovered the identity 
\begin{equation} \label{iden}
m(t_1+t_2+1)=L'(\chi,-1)
\end{equation}
where $\chi$ is the quadratic character of conductor $3$ and 
$L(\chi,s)$ is the Dirichlet series associated to it. Some similar formulas 
can be found in \cite{boyd81b} and \cite{ray87}. The proofs of these 
identities however 
are analytical and do not shed much light on the deeper reasons for this 
phenomenon. \\
This was the situation until Deninger in \cite{deni97} related formulas like 
(\ref{iden}) to Beilinson's conjectures on special values of L-functions.
Assuming these conjectures he found in some way higher dimensional analogues 
of (\ref{iden}) such as 
\begin{equation} \label{such_as}
m(t_1^2 t_2+t_1 t_2^2+t_1 t_2+t_1+t_2)=*L'(E,0)
\end{equation}   
where * denotes (throughout the whole paper) an unknown non-vanishing
rational 
number and $L(E,s)$ is the Hasse-Weil L-function of the elliptic curve 
$E/\mathbb{Q}$ obtained by taking the projective closure of the zero locus 
\begin{equation*} 
t_1^2 t_2+t_1 t_2^2+t_1 t_2+t_1+t_2=0
\end{equation*}  
and adding a suitable origin. \\
This example was the starting point for extensive numerical computations 
done by Boyd (see \cite{boyd98}). He found (numerically) hundreds of
formulas like 
(\ref{such_as}) and similar ones. He also stated a condition under the
presence of which formulas of type (\ref{such_as}) should
hold. Rodriguez-Villegas showed in \cite{rodr98}
that it is precisely this condition that makes it possible to apply 
Beilinson's conjectures. For a special class of polynomials this was
(up to integrality questions) independently done by the author (see
chapter \ref{two}). 

In this paper we set forth the ideas of \cite{deni97} and try to
interprete further parts of the
work of Boyd in the light of Beilinson's conjectures. We succeed in
the following cases: 
\begin{itemize}
\item Boyd observes that some (irreducible) polynomials produce formulas of
  mixed type, i.e. the Mahler measure of such a polynomial is equal to
\begin{equation*} 
*L'(\chi,-1) + *L'(E,0)
\end{equation*}     
for some Dirichlet character $\chi$ and some elliptic curve $E$ over 
$\mathbb{Q}$. For this topic see chapter
\ref{four} and \ref{five}.
\item Another conjecture of Boyd states that no formula of mixed type
  will occur as long as the polynomial is reciprocal. For this problem
  see chapter \ref{one}.
\item Boyd also found formulas of type (\ref{such_as}) where the
  zero locus of the polynomial is of genus two. In those cases the
  elliptic curve E in (\ref{such_as}) turns out to be one of the
  (generally) non-isogenous factors of the Jacobian of the zero locus.
See chapter \ref{three} for an explanation for this rather ``miraculous'' 
occurence. 
\end{itemize}
These notes represent a shortened version of the author's thesis 
\cite{born98}. The reader who wants to see detailed proofs rather than
(just) the underlying basic ideas is referred to this work. \\
Further work in the spirit of \cite{deni97} was done in the following papers: 
In \cite{stan01} the three variable example
\begin{equation*}
m(1+t_1+ t_2+t_3)=*\zeta'(-2) 
\end{equation*}
of Smyth was reduced to the (due to Borel) known Beilinson conjectures for 
$\spec{\mathbb{Q}}$. 
In \cite{den_bes99} an
approach to p-adic analogs of Mahler measures and formulas of type 
(\ref{such_as}) can be found. \\
The author expresses his deep gratitude to his 
``Doktorvater'' Christopher Deninger for introducing him to this area
of research, for making valuable suggestions and last but not least
for being ready to discuss. The author also wants to thank David Boyd
from the University of British Columbia (Vancouver/Canada) for a 
number of interesting and stimulating e-mail correspondence on his  
aforementioned numerical experiments.
\\ 
\\ 
\\ 
{\bf Contents}

\begin{enumerate}
\item Mahler measures and symbols 
\item Boundary maps in $K$-theory 
\item Curves of genus $2$
\item Formulas of mixed type
\item A general philosophy
\end{enumerate}

\section{Mahler measures and symbols}

\label{one}
In this chapter we will rewrite the Mahler measure of a polynomial in two 
variables in a way that allows us to apply $K$-theoretical
methods. The main idea of the following is that in building the Mahler
measure of a polynomial by definition we have to integrate over a
differential form which can be seen as a certain cup product lying in 
some Deligne cohomology group. In our context the main lemma is  
\begin{lemm} \label{ma_is}
For $n\ge 0$ consider elements  
\begin{equation*}
\varepsilon_0,\ldots,\varepsilon_n\in H_{\cal D}^1(X,\mathbb{R}(1))=
\left\{\varepsilon\in {\cal A}^0(X,\mathbb{R})\left| 
\begin{array}{l}
d\varepsilon=\pi_0(\omega), \\
\omega\in\Omega^1_D(\overline{X})
\end{array}\right\}\right.
\end{equation*}
Define a smooth $\mathbb{R}(n)$-valued $n$-form on $X$ by:
\begin{align*}
C_{n+1} &= C_{n+1}(\varepsilon_0,\ldots,\varepsilon_n) \\
 &= 2^n \sum\limits_{i=0}^n (-1)^i \frac{1}{(n+1)!} 
\sum\limits_{\sigma\in {\frak S}_{n+1}}\sgn(\sigma)
\varepsilon_{\sigma_0} \bar{\partial}\varepsilon_{\sigma_1}\dots 
\bar{\partial}\varepsilon_{\sigma_i} \partial\varepsilon_{\sigma_{i+1}}\ldots
\partial\varepsilon_{\sigma_n}.  
\end{align*}  
where ${\frak S}_{n+1}$ is the permutation group of $\{0,\ldots,n\}$.
Then 
\begin{equation*} 
dC_{n+1}=\pi_n(\omega_{n+1})
\end{equation*}
where $\omega_{n+1}=2^{n+1}
\partial\varepsilon_0\wedge\ldots\wedge\partial\varepsilon_n \in
\Omega^{n+1}_D(\overline{X})$ and 
\begin{equation*} 
[C_{n+1}(\varepsilon_0,\ldots,\varepsilon_n)]=
[\varepsilon_0]\cup\ldots\cup [\varepsilon_n]\text{\ in\ }
H_{\cal D}^{n+1}(X,\mathbb{R}(n+1)). 
\end{equation*}
Moreover for all $\sigma\in {\frak S}_{n+1}$
\begin{equation*}
C_{n+1}(\varepsilon_{\sigma_0},\ldots,\varepsilon_{\sigma_n})=\sgn(\sigma)
C_{n+1}(\varepsilon_0,\ldots,\varepsilon_n).
\end{equation*} 
\end{lemm}

\begin{proof}
See \cite{beil86} 2.2 and \cite{deni89} Lemma (7.2).
\end{proof}

Before we can proceed we have to fix some notations. 
Let $0\neq P(t_1,t_2)\in \mathbb{C}[t_1,t_2]$,
\begin{equation} \label{s_let}
P(t_1,t_2)=\sum\limits_{i=0}^n a_i(t_1)t_2^i
\end{equation} 
be irreducible with $a_n\not\equiv 0$. Set 
$i_0:=\min\{i\mid a_i\not\equiv 0\}$ and let $P^*(t_1)$ denote the polynomial 
$a_{i_0}(t_1)$. Assume that $P^*(t_1)=P(t_1,0)$ and that $P^*$ does not vanish
on $S^1$. Let $Z^*(P):=Z(P)\cap (\mathbb{C}^*)^2$. Denote by $A$ the union
of the connected components of dimension $1$ of 
$(S^1\times B)\cap Z^*(P)=(S^1\times B)\cap Z(P)$. Furthermore let 
$A\subset Z^*(P)^{\text{reg}}$. \\
As was remarked by Deninger (see \cite{deni97}) and others using 
Jensen's formula one has 
\begin{equation*}
m(P^*)-m(P)=\int\limits_{S^1} \eta
\end{equation*}   
with the integrable $1$-form on $S^1$ 
\begin{equation*} 
\eta:=\frac{1}{2\pi i}
\sum_{\substack{0\neq b\in\stackrel{\circ}{B} \\
P(t_1,b)=0}} \log|b| \frac{dt_1}{t_1}.
\end{equation*} 
The sum has to be taken with multiplicities of the zeroes 
$0\neq b\in \stackrel{\circ}{B}$ of $P_{t_1}(t):=P(t_1,t)$. The form 
$\eta$ is well defined since $P_{t_1}(t)$ cannot vanish identically due to 
the irreducibility of $P$. \\
Proceeding in the line of \cite{dekn82} Thm.~5.1 we now
``triangulate'' the compact, semi-algebraic set $A$. Set
\begin{align*}
e:[0,1] &\rightarrow S^1 \\
\varphi  &\mapsto e^{2\pi i \varphi}.
\end{align*}
Using implicit functions one can subdivide the interval $I:=[0,1]$ 
into disjoint subintervals $I_k:=[\tau_k,\tau_{k+1}]$ for
$k=0,\ldots,s-1$ and define algebraic germs 
$F_{1k},\ldots,F_{nk}$ of $P$ in a neighbourhood of the arc 
$e(\stackrel{\circ}{I}_k)$ which can be continously extended to the
boundary $\partial e(I_k)$. Therefore we have paths 
\begin{align*}
\gamma_{ik}: I_k &\rightarrow 
S^1\times \mathbb{P}^1(\mathbb{C}) \\
 \varphi &\mapsto (e(\varphi),F_{ik}(e(\varphi)))  
\end{align*}  
which by eventually taking a finer subdivision have the following properties
\begin{enumerate}
\item For a path $\gamma_{ik}$ one and only one of the following
conditions holds: 
\begin{enumerate} 
\item $\gamma_{ik}(\stackrel{\circ}{I_k})\subset 
S^1\times\stackrel{\circ}{B}$. 
\item $\gamma_{ik}(\stackrel{\circ}{I_k})\subset T^2$. 
\item $\gamma_{ik}(\stackrel{\circ}{I_k})\subset (S^1\times B)^c$. 
\end{enumerate}  
\item One has  
\begin{equation*}
A=\bigcup_{\substack{i,k \\ 
\gamma_{ik}(I_k)\subset S^1\times B}}
\gamma_{ik}(I_k).   
\end{equation*}
\item If two paths $\gamma_{ik}$ and $\gamma_{i'k'}$ intersect their 
intersection is contained in 
$\gamma_{ik}(\partial I_k)\cap \gamma_{i'k'}(\partial I_{k'})$. 
\end{enumerate}
Using this construction we have 
\begin{lemm}
Let $P$ satisfy the general assumptions made at the beginning of this 
chapter. Let  
$C_2=C_2(\log|t_1|,\log|t_2|)$ denote the differential form of 
\ref{ma_is}. Then the restriction of $C_2$ to $Z^*(P)^{\text{reg}}$ 
is defined and we have 
\begin{equation*} 
(-2\pi i)\int\limits_{S^1} \eta=
\sum\limits_{k=0}^{s-1} 
\sum_{\substack{i\in \{1,\ldots,n\} \\ \gamma_{ik}(I_k)\subset S^1\times B}} 
\int\limits_{I_k} \gamma_{ik}^* C_2.   
\end{equation*}    
\end{lemm}

\begin{proof}
Since $Z(P)\cap T^2$ doesn't contribute to the integral we have   
\begin{equation*}
(-2\pi i)\int\limits_{S^1} \eta = 
(-2\pi i)\int\limits_0^1 \sum_{\substack{0\neq b\in B \\
P(e(\varphi),b)=0}} \log|b| d\varphi.
\end{equation*}
The above construction gives us   
\begin{align*}
\int\limits_0^1 \sum_{\substack{0\neq b\in B \\
P(e(\varphi),b)=0}} \log|b| d\varphi &= 
\sum\limits_{k=0}^{s-1} \int\limits_{I_k} \sum_{\substack{0\neq b\in B \\
P(e(\varphi),b)=0}} \log|b| d\varphi \\
 &= \sum\limits_{k=0}^{s-1}  
\sum_{\substack{i\in \{1,\ldots,n\} \\ \gamma_{ik}(I_k)\subset S^1\times B}} 
 \int\limits_{I_k} \log|F_{ik}(e(\varphi))| d\varphi.
\end{align*}
We now have to show that 
\begin{equation*}
\int\limits_{I_k} \gamma_{ik}^* C_2=(-2\pi i)
\int\limits_{I_k} \log|F_{ik}(e(\varphi))| d\varphi. 
\end{equation*}
Using $\partial (\log|t_i|)=\frac{1}{2}\frac{d t_i}{t_i}$ 
we get  
\begin{equation*}
C_2(\log|t_1|,\log|t_2|)=\frac{1}{2}\left(
\log|t_1|\frac{d t_2}{t_2}-\log|t_2|\frac{d t_1}{t_1} 
-\log|t_1|\frac{d \bar{t}_2}{\bar{t}_2}
+\log|t_2|\frac{d \bar{t}_1}{\bar{t}_1} \right).
\end{equation*}
According to the definition we have  
\begin{equation*}
\gamma_{ik}(\varphi)=(e(\varphi),F_{ik}(e(\varphi))). 
\end{equation*} 
Computing $\gamma_{ik}^* C_2$ one sees immediately (notice that 
$\log|e(\varphi)|=0$) 
\begin{align*}
\gamma_{ik}^* C_2 &=\frac{1}{2} \left(-\log|F_{ik}(e(\varphi))| 
e(-\varphi) (2\pi i) e(\varphi) d\varphi \right. \\
 & \quad + \left. \log|F_{ik}(e(\varphi))| 
e(\varphi)(-2\pi i) e(-\varphi)d\varphi\right) \\
 &=(-2\pi i) \log|F_{ik}(e(\varphi))|d\varphi.   
\end{align*} 
\end{proof}

\begin{koro}
Using the above notation we get 
\begin{equation*}  
m(P)-m(P^*)=\frac{1}{2\pi i} 
\sum\limits_{k=0}^{s-1} 
\sum_{\substack{i\in \{1,\ldots,n\} \\ \gamma_{ik}(I_k)\subset S^1\times B}} 
\int\limits_{I_k} \gamma_{ik}^* C_2.  
\end{equation*}
\end{koro}

\begin{proof}
Obvious.
\end{proof}

Let us now fix some notations.
Let $K=\mathbb{C}$ or $\mathbb{R}$. For a variety $X$ over
$K=\mathbb{R}$ we get an antiholomorphic involution 
$F_{\infty}$ on $X(\mathbb{C})$. For a complex $\mathbb{C}$-valued form
$\eta$ on $X(\mathbb{C})$ set 
$\overline{F}_{\infty}^*\eta=\overline{F_{\infty}^*\eta}$.   \\
For any  variety $X/K$ and any subgroup $\Lambda\subset\mathbb{C}$
which in case $K=\mathbb{R}$ should in addition satisfy 
$\overline{\Lambda}=\Lambda$ we set  
\begin{align*}
H^n(X/\mathbb{C},\Lambda) &:=H^n_{\text{sing}}(X(\mathbb{C}),\Lambda) && 
\text{and} \\
H^n(X/\mathbb{R},\Lambda) &:=H^n_{\text{sing}}(X(\mathbb{C}),\Lambda)^+
\end{align*} 
where the superscript $+$ denotes taking invariants under the action
of $\overline{F}_{\infty}^*$. A similar definition applies to homology
and to relative situations. \\
Set $\Lambda(n):=(2\pi i)^n \Lambda$. We also need the natural pairing 
\begin{equation*}
\langle .,.\rangle:H^n(X/K,\mathbb{R}(n)) \times 
H_n(X/K,\mathbb{R}(-n)) \rightarrow \mathbb{R}
\end{equation*}
and again similar for relative situations. As a last ingredient we want to 
mention the fact that for $n>\dim X$ we have the equation  
$H_{\cal D}^{i+1}(X/K,\mathbb{R}(n))=
H^i(X/K,\mathbb{R}(n-1))$. \\
Let us now return to our main discussion. 
Connecting the paths $\gamma_{ik}$ in an appropiate way using each
path just one time and
reparametrizing the resulting path we get closed
paths $\chi_{\mu}:[0,1]\rightarrow Z^*(P)^{\text{reg}}$ 
($\mu=0,\ldots,\mu_0$) and paths with
boundary  $\psi_{\nu}:[0,1]\rightarrow Z^*(P)^{\text{reg}}$ 
($\nu=0,\ldots,\nu_0$) satisfying the following properties:  
\begin{enumerate}
\item The boundary points of the $\psi_{\nu}$ are exactly those points
where the number of paths 
$\gamma_{ik}$ running into the point is not equal to the number of paths 
$\gamma_{ik}$ running out of that point. Denote the set of all boundary
points of paths $\psi_{\nu}$ by $R_P$. One has $R_P\subset T^2$. 
\item We have 
\begin{equation*}
A=\bigcup\limits_{\mu=0}^{\mu_0} \chi_{\mu}([0,1]) \cup 
\bigcup\limits_{\nu=0}^{\nu_0} \psi_{\nu}([0,1]).
\end{equation*}   
\end{enumerate}
The paths $\chi_{\mu}$and $\psi_{\nu}$ give us classes 
$[\chi_{\mu}] \in H_1(Z^*(P)^{\text{reg}}/\mathbb{C},\mathbb{Z})$ and 
\linebreak
$[\psi_{\nu}] \in H_1((Z^*(P)^{\text{reg}},R_P)/\mathbb{C};\mathbb{Z})$. 
Considering the $[\chi_{\mu}]$ also as elements of 
$H_1((Z^*(P)^{\text{reg}},R_P)/\mathbb{C};\mathbb{Z})$ we set 
\begin{equation*}
[A]:=\sum\limits_{\mu=0}^{\mu_0} [\chi_{\mu}]+ 
\sum\limits_{\nu=0}^{\nu_0} [\psi_{\nu}].
\end{equation*}  
Now note that the restriction of the $1$-form $C_2$ to 
$Z^*(P)^{\text{reg}}$ is closed therefore defining a cohomology class 
$[C_2]\in H^1(Z^*(P)^{\text{reg}},\mathbb{R}(1))$. 
Since the restriction of $C_2$ to $R_P\subset T^2$ is zero we may
also view it as defining a relative cohomology class 
$[C_2]\in H^1((Z^*(P)^{\text{reg}},R_P)/\mathbb{C};\mathbb{R}(1))$. \\
Using de Rham theorem it is not hard to show the following claim:

\begin{satz}  \label{following}
Let $P$ be as above. There is a class  
\begin{equation*}
[A]\otimes (2\pi i)^{-1}\in 
H_1((Z^*(P)^{\text{reg}},R_P)/\mathbb{C};\mathbb{Z}(-1)),
\end{equation*}
satisfying  
\begin{equation} \label{cohomform}
m(P)-m(P^*)=\langle [C_2],[A]\otimes (2\pi i)^{-1}\rangle.
\end{equation}
\end{satz}

\begin{beme}
For polynomials $P\in\mathbb{C}[t_1,\ldots,t_n]$ such that $Z^*(P)$ is smooth 
and does not intersect $T^n$ a cohomological generalization of formula 
(\ref{cohomform}) was given in \cite{den_bes99} Proposition 2.2.
\end{beme}

\begin{koro} \label{claim}
Let $P\in \mathbb{Q}[t_1,t_2]$ be as above and assume  
$R_P=\emptyset$. Then we have  
\begin{equation*} 
[A]\otimes (2\pi i)^{-1}\in 
H_1(Z^*(P)^{\text{reg}}/\mathbb{R},\mathbb{Z}(-1)) 
\end{equation*} 
and
\begin{equation*} 
m(P)-m(P^*)=\langle r_{\cal D}(\{t_1,t_2\})
,[A]\otimes (2\pi i)^{-1}\rangle,
\end{equation*}
where $\{t_1,t_2\}\in H_{\cal M}^2(Z^*(P)^{\text{reg}},\mathbb{Q}(2))$
and  
\begin{equation*}
r_{\cal D}:H_{\cal M}^2(Z^*(P)^{\text{reg}},\mathbb{Q}(2)) \rightarrow 
H_{\cal D}^2(Z^*(P)^{\text{reg}},\mathbb{R}(2))
\end{equation*}
denotes as usual the regulator.  
\end{koro}

\begin{proof} 
We calculate 
\begin{align*}
[C_2] &=[\log|t_1|]\cup [\log|t_2|]  && \text{see \ref{ma_is}} \\
 &=r_{\cal D}(t_1) \cup r_{\cal D}(t_2) \\ 
 &=r_{\cal D}(\{t_1,t_2\})
\end{align*} 
using the compatibility of the regulator with respect to cup products and the
fact that $r_{\cal D}(t_i)=\log|t_i|$.  
\end{proof}

Our assumptions imposed on the polynomial $P$ at the beginning of the
chapter are very restrictive. The following lemma allows us to weaken
those conditions. But before doing so we need another notation. 
Let $A=\bigl(\begin{smallmatrix}a&b\\c&d\end{smallmatrix}\bigr) 
\in \text{GL}_2(\mathbb{Z})$ and define
\begin{align*}
\phi_A : (\mathbb{C}^*)^2 &\rightarrow (\mathbb{C}^*)^2 \\
  (t_1,t_2) & \mapsto (t_1^a t_2^c,t_1^b t_2^d).
\end{align*} 
 
\begin{lemm} \label{fine}
Let $0 \neq P(t_1,t_2)\in \mathbb{C}[t_1,t_2]$ written as in
(\ref{s_let}) and assume $P^*(t_1)=a_0(t_1)$. Let   
$Z^*(P)^{\text{sing}}=\{(z_1^{(i)},z_2^{(i)})\mid i=1,\ldots,r\}$ be
the finite set of singularities of $Z^*(P)$.  Assume 
$Z^*(P)^{\text{sing}}\cap T^2=\emptyset$. 
Then there exists an 
$A\in \text{GL}_2(\mathbb{Z})$ so that 
$Q(t_1,t_2):=(t_1 t_2)^{\deg(a_0)}\cdot (\phi_A^*P)(t_1,t_2)$
satisfies the following conditions: 
\begin{enumerate}
\item $m(Q)=m(P)$.
\item $Q(t_1,t_2)\in \mathbb{C}[t_1,t_2]$.
\item $Q^*(t_1)=Q(t_1,0)$ and $Q^*(t_1)$ is equal to the leading
coefficient of $P^*$.  
\item $Z^*(Q)^{\text{sing}}\cap (S^1\times B)=
(\phi_A)^{-1} (Z^*(P)^{\text{sing}})\cap (S^1\times B)=\emptyset$.  
\item If $P$ is irreducible, so is $Q$.
\item If $P$ is reciprocal, so is $Q$.
\end{enumerate}
\end{lemm}

\begin{proof}
Everything is obvious except of 4.: Choose 
$m_1\in\mathbb{N}$ with $m_1\ge \deg(a_i)$ for all $i\ge 1$. 
Let $\lambda_j^{(i)}:=\log|z_j^{(i)}|$. Choose in addition 
$m_2\in\mathbb{N}$ such that
\begin{equation} \label{that}
(m_2+1)\lambda_1^{(i)}+\lambda_2^{(i)}\neq 0
\end{equation}
for all $i=1,\ldots,r$. Let $m:=\max\{m_1,m_2\}$ and    
\begin{equation*}
A:=
\left(
\begin{array}{cc}
 -1 & m \\
 -1 & m+1  
\end{array}
\right)\in\text{GL}_2(\mathbb{Z}).
\end{equation*}
It is now easily seen that 4. holds for $A$ defined as above. 
\end{proof}

\begin{beme}
Roughly speaken the last lemma says that by changing to the polynomial $Q$ 
we can get rid of singularities in $S^1\times B\cap Z(P)$ as long as we assume 
$Z^*(P)^{\text{sing}}\cap T^2=\emptyset$. For two variable polynomials this 
means that we no longer need the condition $A\subset Z^*(P)^{\text{reg}}$ 
which origins in ``Assumptions 3.2'' from \cite{deni97}.
\end{beme}

\begin{koro} \label{reci}
Let $P\in \mathbb{Q}[t_1,t_2]$ be absolutely irreducible and 
reciprocal. In addition assume $Z^*(P)^{\text{sing}}\cap T^2=\emptyset$. 
Denote by  $\lambda$ the leading coefficient of $P^*(t_1)$.
Then there exists a class
$[A]_P \otimes (2\pi i)^{-1}\in H_1(Z^*(P)^{\text{reg}},\mathbb{Z}(-1))$, 
with  
\begin{equation*} 
m(P)=\log|\lambda|\pm \langle r_{\cal D}(\{t_1,t_2\})
,[A]_P\otimes (2\pi i)^{-1}\rangle.
\end{equation*}
\end{koro}

\begin{proof}
We want to use \ref{fine} to reduce to the situation of \ref{claim}. 
To do so we still have to show that $R_Q=\emptyset$.
\ref{fine} 3. shows us that $Q$ does not
vanish at points of the form $(\cdot,0)$.  
Take $(e(\varphi),r e(\psi))\in S^1\times \mathbb{C}^*$ with  
$Q(e(\varphi),r e(\psi))=0$. Due to \ref{fine} 6. $Q$ is reciprocal
and we have   
\begin{equation*}
Q(e(-\varphi),\frac{1}{r} e(-\psi))=0.
\end{equation*} 
Applying complex conjugation gives us  
\begin{equation*}
Q(e(\varphi),\frac{1}{r} e(\psi))=0.
\end{equation*}   
Suppose we have a path in $Z^*(Q)$ coming from $S^1\times B$ 
intersecting $T^2$ 
in a point  and then leaving $S^1\times B$. The above calculation
then shows us that another path in $Z^*(Q)$ comes from the outside of 
$S^1\times B$ intersects $T^2$ in the same point as above and runs into 
$S^1\times B$. If we have two pathes both running in $T^2$ we can
discard one (because it doesn't contribute to the integral we are 
considering). From these observations we get $R_Q=\emptyset$. 
Let $[A]$ be the class from \ref{following} built with respect to our 
polynomial $Q$. Using the isomorphism 
$\phi_A:Z^*(Q)^{\text{reg}}\rightarrow Z^*(P)^{\text{reg}}$ 
from \ref{fine} we set $[A]_P:=(\phi_A)_* [A]$. 
Applying \ref{claim} and \ref{fine} we conclude the proof.    
\end{proof}

\begin{beme}
\begin{enumerate}
\item Starting from his numerical experiments in \cite{boyd98} Boyd has
conjectured that for reciprocal polynomials with zero locus of genus
$1$ one has always formulas analogous to (\ref{such_as}). Our last
result explains this in some way: the fact that the polynomial in
question is reciprocal allows us to deal with absolute homology or
cohomology classes rather than with relative ones. 
\item The assumption $Z^*(P)^{\text{sing}}\cap T^2=\emptyset$ seems to
be crucial as the following example shows: Set  
\begin{equation*}
P(t_1,t_2)=(t_1^2+t_1+1)t_2^2+(t_1^4-t_1^3-6 t_1^2-t_1+1)t_2+t_1^2
(t_1^2+t_1+1)
\end{equation*}
$Z(P)$ is of genus $1$ and $(S^1 \times B)\cap Z(P)$ is a closed 
path on which the singular points $(-1,1),(1,1)$  lie. 
Boyd gets numerically the unexpected formula 
\begin{equation*}
m(P)=* L'(\chi_1,-1)+* L'(\chi_2,-1) 
\end{equation*}  
where $\chi_1$ and $\chi_2$ are two Dirichlet characters. 
Seemingly one has to build the normalization of the curve $Z(P)$ first.
\end{enumerate}
\end{beme}

\begin{koro} \label{empt}
Let $P\in \mathbb{Q}[t_1,t_2]$ be absolutely irreducible and 
assume that
$Z(P)\cap T^2=\emptyset$. 
Denote by $\lambda$ the leading coefficient of $P^*(t_1)$. 
There exists a class
$[A]_P\otimes (2\pi i)^{-1} \in H_1(Z^*(P)^{\text{reg}},\mathbb{Z}(-1))$, 
such that 
\begin{equation*} 
m(P)=\log|\lambda|\pm \langle r_{\cal D}(\{t_1,t_2\})
,[A]_P\otimes (2\pi i)^{-1}\rangle.
\end{equation*}
\end{koro}

\begin{proof}
Again we use \ref{claim} and \ref{fine}.     
\end{proof}

\section{Boundary maps in $K$-theory}

\label{two}
According to our general policy we want to use Beilinson's 
conjectures and theorems \ref{claim}, \ref{reci} and \ref{empt} to 
produce formulas like (\ref{such_as}). Since Beilinson's conjectures deal with 
projective, smooth varieties over $\mathbb{Q}$ we need to know that
our symbol $\{t_1,t_2\}$ already lies in the motivic cohomology of the
projective, smooth model of our initial curve, 
i.e. we need to know that our symbol
vanishes under the tame symbol. \\ 
We have to fix some notations. Let 
\begin{equation*}
P(t_1,t_2)=\sum\limits_{k_1,k_2} 
\alpha_{k_1,k_2} t_1^{k_1}t_2^{k_2}
\in \mathbb{Q}[t_1,t_2]
\end{equation*}
an absolute irreducible polynomial. Denote by $Z(P)$ the algebraic
variety 
over $\mathbb{Q}$ defined by the equation $P=0$. Let $C$ denote the
normalization of the projective closure of $Z(P)$. Consider $t_1,t_2$ 
as rational functions on $C$. Set 
$S:=\spp(\divi(t_1))\cup \spp(\divi(t_2))$ and $U:=C-S$. \\
Define the Newton polygon ${\cal N}(P)$ of our polynomial $P$ to be
the convex hull of the set 
$\{(k_1,k_2)\mid \alpha_{k_1,k_2}\neq 0 \}$ in 
$\mathbb{R}^2$. For a side $F$ of ${\cal N}(P)$ we parametrize the points
of $F\cap \mathbb{Z}^2$ clockwise in such a manner that 
$(k_1^0,k_2^0),\ldots,(k_1^l,k_2^l)$ are the consecutive lattice
points of $F$. One can attach to every side $F$ of 
${\cal N}(P)$ a one-variable polynomial 
\begin{equation*}
P_{F}(t):=\sum_{i=0}^l \alpha_{k_1^i,k_2^i} t^i\in\mathbb{Q}[t]. 
\end{equation*} 
Boyd calls a polynomial $P$ tempered if all $P_F$ for all sides $F$ of 
${\cal N}(P)$ have only roots of unity as zeroes. \\    
Let us now return to K-theory. Obviously one has 
$\{t_1,t_2\}\in H^2_{\cal M}(U,\mathbb{Q}(2))$. 
As mentioned above we
want to know under which assumptions 
$\{t_1,t_2\}\in H^2_{\cal M}(C,\mathbb{Q}(2))$ 
holds. The following
theorem gives the answer:

\begin{satz} \label{swer}
With notations as above the following two conditions are equivalent
\begin{enumerate}
\item[(1)] $\{t_1,t_2\}\in  H^2_{\cal M}(C,\mathbb{Q}(2))$.
\item[(2)] $P$ is tempered.
\end{enumerate}
\end{satz}

\begin{proof}
The general case is due to 
Rodriguez-Villegas (see \cite{rodr98} chapter 8). 
For the special form $P(t_1,t_2)=A(t_1)t_2^2+B(t_1)t_2+C(t_1)$ of 
polynomials considered by Boyd the proof is an easy but tedious calculation. 
At first one 
has to calculate the divisors of $t_1$ and $t_2$ as rational functions 
on $C$. After doing so one can determine the tame symbol
\begin{equation*} 
K_2(\mathbb{Q}(C))\otimes\mathbb{Q}
\overset{\partial=\coprod\partial_p}{\rightarrow}
\coprod\limits_{p\in C(\bar{\mathbb{Q}})} 
\mathbb{Q}(p)^*\otimes\mathbb{Q}.
\end{equation*}
where
\begin{equation*} 
\partial_p(\{f,g\})=\left[
(-1)^{\ord_p(f)\ord_p(g)}
\frac{f^{\ord_p(g)}}{g^{\ord_p(f)}}\right](p)\otimes 1.
\end{equation*}
It shows up that while $p$ runs over $p\in\spp(\divi(t_1))\cup
\spp(\divi(t_2))$
there always exists a zero $\zeta_p$ of a polynomial $P_F$ for a side $F$
of ${\cal N}(P)$ such that  
\begin{equation*} 
\partial_p(\{t_1,t_2\})=\zeta_p \otimes 1.
\end{equation*}
This takes care of the implication (2) $\Rightarrow$ (1). 
One also notes that for every side $F$ of ${\cal N}(P)$ and every zero 
$\zeta$ of the polynomial $P_F$ attached to the side there is a 
$p_{\zeta}\in\spp(\divi(t_1))\cup\spp(\divi(t_2))$ such that
\begin{equation*} 
\partial_{p_{\zeta}}(\{t_1,t_2\})=\zeta \otimes 1.
\end{equation*}
This gives us the implication (1) $\Rightarrow$ (2). 
\end{proof}

Let us now assume that $C$ is an elliptic curve over $\mathbb{Q}$, i.~e.~is
of genus $1$ and has got a $\mathbb{Q}$-rational point. Beilinson's 
conjectures deal with 
$H^2_{\cal M}(C,\mathbb{Q}(2))_{\mathbb{Z}} \subset 
H^2_{\cal M}(C,\mathbb{Q}(2))$.
So even if our symbol is already an
element of $H^2_{\cal M}(C,\mathbb{Q}(2))$ it has to overcome another
obstruction, the so
called integral obstruction $\delta$ ($\cal C/\mathbb{Z}$ denotes the 
minimal regular model of $C$): 
\begin{displaymath}
0 \rightarrow
H^2_{\cal M}(C,\mathbb{Q}(2))_{\mathbb{Z}}\rightarrow  
H^2_{\cal M}(C,\mathbb{Q}(2))
\overset{\delta=\coprod\delta_p}{\rightarrow}
\coprod\limits_{p} 
K_1'({\cal C}_p)\otimes\mathbb{Q} \rightarrow \ldots.
\end{displaymath}
The following theorem gives us an example for a whole family of curves
where the integral obstruction of a certain symbol vanishes (enabling
us to produce a formula like (\ref{such_as})).

\begin{satz}
Take the following family of polynomials from $\mathbb{Z}[t_1,t_2]$:
\begin{equation*}
P_k(t_1,t_2):=t_1t_2^2+ (t_1^2+kt_1+1)t_2+t_1.
\end{equation*}
Assume $k\in \mathbb{Z}-\{0,\pm 4\}$. Then we have 
\begin{enumerate}
\item The zero locus $Z(P_k)$ is birationally equivalent to an  
elliptic curve $C_k$ over $\mathbb{Q}$.
\item Assuming Beilinson's conjectures for elliptic curves we get   
\begin{equation*} 
m(P_k)=*L'(C_k,0)
\end{equation*} 
\end{enumerate}
\end{satz}

\begin{proof}
Assume $k\in \mathbb{Z}-\{0,\pm 4\}$. Let $C=C_k$ be the elliptic curve 
defined by the Weierstrass equation
\begin{equation} \label{weierequation}
y^2+kxy+ky=x^3+x^2.
\end{equation} 
The map 
\begin{align*}
Z(P_k) &\rightarrow C \\
(t_1,t_2) &\mapsto (k(t_1+t_2)^{-1},-k(t_1+t_2)^{-2}t_1(t_1+t_2+k)) 
\end{align*}
establishes a birational equivalence between the two curves thereby taking 
care of our first claim. 
Now using \ref{reci} we show 
\begin{equation} \label{we_show}
m(P_k)=\pm\langle r_{\cal D}\{t_1,t_2\},\gamma_k\rangle
\end{equation} 
with $\{t_1,t_2\}\in H^2_{\cal M}(Z^*(P_k)^{\text{reg}},\mathbb{Q}(2))$ and 
$\gamma_k\in H_1(Z^*(P_k)^{\text{reg}},\mathbb{Z}(-1))$. 
Theorem \ref{swer} gives us  
\begin{equation*}
\{t_1,t_2\}\in 
H^2_{\cal M}(C,\mathbb{Q}(2)).
\end{equation*} 
Now we have to calculate the integral obstruction of the symbol 
$\{t_1,t_2\}$. If the reduction at $p$ ist not split multiplicative we 
have $K_1'({\cal C}_p)\otimes\mathbb{Q}=0$. Therefore we confine ourselves to 
the case of split multiplicative reduction at $p$. Then ${\cal C}_p$ is a 
N\'eron $N$-gon. For the divisors 
\begin{equation*}
\divi(t_1)=\sum\limits_i a_i (X_i) \text{ and }
\divi(t_2)=\sum\limits_j b_j (Y_j) 
\end{equation*}
set 
\begin{equation*}
d_{t_1}(\nu)=\sum\limits_i a_i d_{X_i}(\nu),
\end{equation*}
where 
\begin{equation*}
d_{X_i}(\nu)=
\begin{cases}
1 & \text{if $X_i$ reduces to the $\nu$-th side of the N-gon} \\
0 & \text{else.}
\end{cases}
\end{equation*}
Using this notation we have the following formula which is due to Schappacher 
and Scholl (see \cite{scha_scho91} chapter 3):
\begin{equation} \label{chapterthree}
\delta_p(\{t_1,t_2\})=\pm\frac{1}{3N}
\sum_{\substack{\mu\in \mathbb{Z} \\ \nu\in \mathbb{Z}}}
d_{t_1}(\mu)d_{t_2}(\nu+\mu)
B_3\left(\left\langle\frac{\nu}{N} \right\rangle\right).
\end{equation}
Here $B_3(x)=x^3-\frac{3}{2}x^2+\frac{1}{2}x$ (the third Bernoulli 
polynomial) and 
$\langle\frac{\nu}{N}\rangle\equiv\frac{\nu}{N}\bmod\mathbb{Z}$ subject to 
the condition $\langle\frac{\nu}{N}\rangle\in [0,1[$. \\
Returning to our special situation let us remark that for the curve $C$ we 
get $c_4=k^4-16k^2+16$ and $\Delta=k^2(k-4)(k+4)$. Furthermore on $C$ our 
divisors read as 
\begin{align*}
\divi(t_1)&=(O)+(Q)-(2Q)-(3Q) && \text{and} \\
\divi(t_2)&=(O)-(Q)-(2Q)+(3Q)
\end{align*} 
where $O$ is the origin of $C$ and $Q=(0,0)$. Clearly we have $2Q=(-1,0)$,
$3Q=(0,-k)$ and $4Q=O$. \\
In computing the reduction of $C$ and the four points $O,Q,2Q,3Q$ on it at 
a prime $p$ let us first assume $p\ge 3$. Clearly the inequation 
$v_p(\Delta)>0$ is then equivalent to having $p|k$, $p|(k-4)$ or $p|(k+4)$. 
Say $p|k$. 
Since $v_p(c_4)=0$ our Weierstrass equation (\ref{weierequation}) 
is minimal and it has multiplicative reduction at $p$ which in addition 
we assume to be split multiplicative. The reduced equation
\begin{equation*}
\tilde{C}: y^2=x^3+x^2
\end{equation*}    
has got the singular point $(0,0)$. Now let  
$\tilde{C}_{\text{ns}}(\mathbb{F}_p)$ denote the set of non-singular points 
of $\tilde{C}(\mathbb{F}_p)$ i.e.~in our setting  
$\tilde{C}_{\text{ns}}(\mathbb{F}_p)=\tilde{C}(\mathbb{F}_p)-\{(0,0)\}$.
Furthermore set 
$C_0(\mathbb{Q}_p) = \{P\in C(\mathbb{Q}_p)\mid \tilde{P} \in 
\tilde{C}_{\text{ns}}(\mathbb{F}_p)\}$. Clearly we have 
\begin{equation} \label{clearly_we_have2}
\ord_{C(\mathbb{Q}_p)/C_0(\mathbb{Q}_p)}(mQ) =
\begin{cases}
1 &\text{for $m=0,2$} \\
2 & \text{for $m=1,3$}.
\end{cases}
\end{equation}
Let us now consider the following well kown fact on the the N\'eron model 
${\cal N}/\mathbb{Z}_p$ of the 
elliptic curve $C/\mathbb{Q}_p$: set 
$\tilde{\cal N}={\cal N}\times_{\mathbb{Z}_p } \mathbb{F}_p$
and let $\tilde{\cal N}^0$ denote the component of the identity in the 
group variety $\tilde{\cal N}$. Then under the 
identification ${\cal N}(\mathbb{Z}_p)\cong C(\mathbb{Q}_p)$ we get
\begin{equation*} 
\tilde{\cal N}(\mathbb{F}_p)/\tilde{\cal N}^0(\mathbb{F}_p)
\cong C(\mathbb{Q}_p)/C_0(\mathbb{Q}_p).
\end{equation*}
Using this fact and (\ref{clearly_we_have2}) we have 
\begin{align*}
d_O(\nu) &= d_{2Q}(\nu) && \text{ and} \\
d_Q(\nu) &= d_{3Q}(\nu)
\end{align*} 
and hence 
\begin{equation*}
d_{t_1}(\nu)=1\cdot d_O(\nu)+1\cdot d_Q(\nu)-1\cdot d_{2Q}(\nu)
-1\cdot d_{3Q}(\nu)=0.
\end{equation*} 
The cases $p\ge 3$, $p|(k-4)$ and $p\ge 3$, $p|(k+4)$ proceed 
in a very similar line and are therefore omitted. In the case  $p=2$ the 
reduction is additive for $0<v_2(k)<4$. For $v_2(k)\ge 4$ one changes to a 
minimal Weierstrass equation and concludes almost verbatim like above. \\
After all we get  
\begin{equation*}
\{t_1,t_2\}\in 
H^2_{\cal M}(C_k,\mathbb{Q}(2))_{\mathbb{Z}}.
\end{equation*} 
A standard inequality from the theory of Mahler measures 
shows us that $m(P_k)\neq 0$ and therefore by 
(\ref{we_show}) 
$\{t_1,t_2\}\neq 0$ and $\gamma_k\neq 0$. Now using Beilinson's
conjectures for elliptic curves we get  
\begin{equation*} 
m(P_k)=*L'(C_k,0).
\end{equation*}  
\end{proof}

\begin{beme}
\begin{enumerate}
\item Rodriguez-Villegas has announced that he found
theoretical arguments for the vanishing of the integral obstruction of
certain symbols. 
\item Boyd has also given several examples for which it is possible to
prove a formula like (\ref{such_as}) rigorously, i.e without assuming
the validity of Beilinson's conjectures. In these examples the
elliptic curves in consideration have got CM. This crucial fact allows
one to apply methods from \cite{de_wi88}. For the details see 
\cite{born98} chapter 5.5.
\end{enumerate}
\end{beme}

\section{Curves of genus $2$}

\label{three}
In \cite{boyd98} Boyd has also computed lots of examples where curves of
genus $2$ occur. Set for example
\begin{equation*}
P(t_1,t_2):=(t_1^2+t_1+1)t_2^2+t_1(t_1+1)t_2+t_1(t_1^2+t_1+1).
\end{equation*}
Let $C$ again be the normalization of the projective closure of $Z(P)$.
The curve $C$ has genus $2$. Its Jacobian $J(C)$ is reducible, 
i.~e.~it is isogenous to a product of two elliptic curves. 
Numerically it seems that 
\begin{equation*}
m(P)=* L'(E,0),
\end{equation*}
where $E$ is one of the above factors of the Jacobian. It is by no
means clear why the Mahler measure ``ignores'' the other elliptic
curve. In this chapter we exhibit the $K$-theoretical reasons for this
behaviour. \\
First we have to fix notations. 
Let $P(t_1,t_2):=A(t_1)t_2^2+B(t_1)t_2+C(t_1)\in \mathbb{Z}[t_1,t_2]$
a tempered, reciprocal polynomial. 
Set $D(t_1):=B(t_1)^2-4A(t_1)C(t_1)$. Assume that 
$D(t_1)=(t_1+1)^{2r}\tilde{D}(t_1)$, where $r\in\mathbb{N}$
and $\tilde{D}\in \mathbb{Z}[t_1]$ is of degree $5$ or $6$ with
non-vanishing discriminant. Furthermore let $s$ be the unique natural
number subject to the condition  
\begin{equation*}
P(t_1,t_2)=t_1^st_2^2P\left(\frac{1}{t_1},\frac{1}{t_2}\right).
\end{equation*}
Assume finally that $s=3+r$. This is in some way a natural assumption
because it follows easily from the above assumptions that we always have 
$s\ge 3+r$.  \\
One defines easily a birational equivalence from $Z(P)$ to the curve 
\linebreak 
$Z(y^2-\tilde{D}(x))$. Furthermore it can be shown that   
$t_1^6\tilde{D}(\frac{1}{t_1})=\tilde{D}(t_1)$. Using the transformation
\begin{align*}
x &= \frac{S+1}{S-1} \\
y &= \frac{T}{(S-1)^3}
\end{align*}
(see \cite{cas_fly96} p.~160) we can get our curve 
$Z(y^2-\tilde{D}(x))$ birational
equivalent to a curve with model $T^2=Q(S^2)$, where 
$Q(z):=c_3z^3+c_2z^2+c_1z+c_0\in \mathbb{Q}[z]$ with non-vanishing 
discriminant and $c_0 c_3\neq 0$. Let $C$ be the normalization of the
projective closure of $T^2=Q(S^2)$. 
Define $\theta$ to be the symbol
$\{t_1,t_2\}$ on $Z(P)$ transformed to our current model
$T^2=Q(S^2)$. Since $P$ is tempered we have 
$\theta\in H^2_{\cal M}(C,\mathbb{Q}(2))$. \\
In this situation we use Theorem 14.1.1 from \cite{cas_fly96}. 
As in the proof
of the theorem we define two elliptic curves $E_1: w^2=Q(z)$ and 
$E_2: w^2=z^3Q(\frac{1}{z})$ and two Galois coverings 
$\varphi_1:C\rightarrow E_1$, $(S,T)\mapsto (S^2,T)$
and $\varphi_2:C\rightarrow E_2$, $(S,T)\mapsto (S^{-2},TS^{-3})$. The
Galois group $\gal(\mathbb{Q}(C)/\varphi_1^*\mathbb{Q}(E_1))$ is 
generated by 
\begin{align*}
\tau_1:\mathbb{Q}(C) &\rightarrow \mathbb{Q}(C) \\
 S &\mapsto -S \\
 T &\mapsto T.
\end{align*}  
One projection formula from motivic cohomology reads therefore 
$\varphi_1^*\circ\varphi_{1,*} = \id+\tau_1$. 
Since we assume $P$ to be tempered we get 
\begin{align*}
\theta &=\left\{
\frac{S+1}{S-1},
\frac{\frac{(2S)^rT}{(S-1)^{3+r}}
-B\left(\frac{S+1}{S-1}\right)}{2A\left(\frac{S+1}{S-1}\right)}
\right\} \\ 
&=\left\{\frac{S+1}{S-1},
\frac{\frac{(2S)^rT}{(S-1)^{3+r}}
-B\left(\frac{S+1}{S-1}\right)}{2\alpha}\right\}
\end{align*} 
where $\alpha$ denotes the leading coefficient of the polynomial $A$. 
In the above computation we have used the fact $\{x, x-\zeta\}=0$ for 
$\zeta$ a root of unity.
Now applying our assumptions $P$ reciprocal and $s=3+r$ we compute
\begin{align*}
\tau_1\theta &= \left\{\frac{S-1}{S+1},
\frac{-\frac{(2S)^rT}{(S+1)^{3+r}}
-B\left(\frac{S-1}{S+1}\right)}{2\alpha}\right\} \\ 
&= \left\{\frac{S-1}{S+1},
\frac{-\frac{(2S)^rT}{(S+1)^{3+r}}
-\left(\frac{S-1}{S+1}\right)^s
B\left(\frac{S+1}{S-1}\right)}{2\alpha}\right\} \\ 
&= \left\{
\frac{S-1}{S+1},
\left(\frac{S-1}{S+1}\right)^s
\frac{-\frac{(2S)^rT}{(S-1)^{3+r}}
-B\left(\frac{S+1}{S-1}\right)}{2\alpha}\right\}\\
&= \left\{
\frac{S-1}{S+1},
\frac{\frac{(2S)^rT}{(S-1)^{3+r}}
+B\left(\frac{S+1}{S-1}\right)}{-2\alpha}\right\}. 
\end{align*}
Let us now denote by $\gamma$ the leading coefficient of the polynomial $C$.
Since $P$ is tempered we clearly have $\alpha=\pm\gamma$. Further using 
the fact 
\begin{equation*}
\frac{\frac{(2S)^r T}{(S-1)^{3+r}}
+B\left(\frac{S+1}{S-1}\right)}{-2C\left(\frac{S+1}{S-1}\right)}=
\frac{2A\left(\frac{S+1}{S-1}\right)}
{\frac{(2S)^r T}{(S-1)^{3+r}}
-B\left(\frac{S+1}{S-1}\right)}
\end{equation*}
we proceed in doing our computation 
\begin{align*}
\tau_1\theta & = \left\{
\frac{S-1}{S+1},
\frac{\frac{(2S)^r T}{(S-1)^{3+r}}
+B\left(\frac{S+1}{S-1}\right)}{-2\gamma}\right\} \\
& = \left\{
\frac{S-1}{S+1},
\frac{\frac{(2S)^r T}{(S-1)^{3+r}}
+B\left(\frac{S+1}{S-1}\right)}{-2C\left(\frac{S+1}{S-1}\right)}
\right\} \\
& = \left\{
\frac{S-1}{S+1},
\frac{2A\left(\frac{S+1}{S-1}\right)}{\frac{(2S)^r T}{(S-1)^{3+r}}
-B\left(\frac{S+1}{S-1}\right)}\right\} \\
& = \left\{
\frac{S+1}{S-1},
\frac{\frac{(2S)^r T}{(S-1)^{3+r}}
-B\left(\frac{S+1}{S-1}\right)}{2A\left(\frac{S+1}{S-1}
\right)} \right\} \\
& = \theta.
\end{align*}
At last we get 
\begin{equation*}
\varphi_1^*\circ\varphi_{1,*}\theta=2\theta.  
\end{equation*}    
An analogous computation can be done using the Galois covering 
$\varphi_2$. The result is 
\begin{equation*}
\varphi_2^*\circ\varphi_{2,*}\theta=0  
\end{equation*}    
and therefore $\varphi_{2,*}\theta=0$. \\
The above discussion is the main ingredient in the following theorem. 
\begin{satz}
Let the assumption of the above discussion apply. Furthermore 
\begin{enumerate}
\item let $P$ be absolute irreducible,
\item let the coefficients of $P$ associated to the extremal points of 
${\cal N}(P)$ have absolute value $1$,
\item let $Z^*(P)^{\text{sing}}\cap T^2=\emptyset$ and  
\item let $m(P)\neq 0$.
\end{enumerate} 
In addition assume 
\begin{equation*}
\varphi_{1,*}\theta\in H^2_{\cal M}(E_1,\mathbb{Q}(2))_{\mathbb{Z}}
\end{equation*}
and the validity of the Beilinson conjectures for elliptic curves. 
One has 
\begin{equation*}
m(P)=* L'(E_1,0).
\end{equation*}
\end{satz}
 
\begin{proof}
Using \ref{reci}   
and transforming everything into the model $C$ we get  
\begin{equation*} 
m(P)=\pm \langle r_{\cal D}(\theta)
,\gamma\rangle
\end{equation*}
where $\gamma\in H_1(C,\mathbb{Z}(-1))$. 
We conclude 
\begin{align*} 
m(P) &=\pm \langle r_{\cal D}(\theta),\gamma\rangle \\
 &= \pm \frac{1}{2}\langle r_{\cal D}(2\theta),\gamma\rangle \\
 &=  \pm\frac{1}{2} \langle 
r_{\cal D}((\varphi_1^*\circ\varphi_{1,*})\theta),\gamma\rangle \\
 &= \pm\frac{1}{2} \langle 
\varphi_1^* r_{\cal D}(\varphi_{1,*}\theta),\gamma\rangle \\
 &= \pm\frac{1}{2} \langle 
r_{\cal D}(\varphi_{1,*}\theta),\varphi_{1,*}\gamma\rangle.
\end{align*}
Since $m(P)\neq 0$ we have $\varphi_{1,*}\gamma\neq 0$ and 
$\varphi_{1,*}\theta\neq 0$.  Using Beilinson conjectures we finally get 
\begin{equation*} 
m(P)= * L'(E_1,0).
\end{equation*}
\end{proof}

\section{Formulas of mixed type}

\label{four}
Another interesting example of Boyd is given by 
\begin{equation} \label{Boyd_is_given_by}
P(t_1,t_2):=(t_1^2+1)^2 t_2^2+2 t_1 t_2+1.
\end{equation}
Numerically evidence suggests
\begin{equation} \label{evidence_suggests}
m(P)=* L'(E,0)+* L'(\chi,-1)
\end{equation}
to be true where $E$ is an elliptic curve over $\mathbb{Q}$ (defined as
usual) and 
$\chi$ is the non-trivial Dirichlet character of  
$\mathbb{Z}/3\mathbb{Z}$. \\
Using the notation of chapter \ref{one} we have
\begin{equation*}
R_P=\{(\zeta_3,-\zeta_3^{-1}),(\zeta_3^{-1},-\zeta_3),
(\zeta_6,-\zeta_6^{-1}),(\zeta_6^{-1},-\zeta_6)\}
\end{equation*}
where $\zeta_3=\exp(\frac{2\pi i}{3})$ and
$\zeta_6=\exp(\frac{\pi i}{3})$. 
Theorem \ref{following} gives us 
\begin{equation} \label{th_gives_us}
m(P)=\langle [C_2],[A]\otimes (2\pi i)^{-1}\rangle
\end{equation}
for a certain class $[A]\otimes (2\pi i)^{-1}\in 
H_1((Z^*(P)^{\text{reg}},R_P)/\mathbb{R};\mathbb{Z}(-1))$. \\
Let $E$ denote as usual the non-singular projective model of 
$Z(P)$. This is an elliptic curve defined over $\mathbb{Q}$. 
Consider $R:=R_P$ and $Z^*(P)^{\text{reg}}$ as subvarieties of $E$. \\
Set $R_{/\mathbb{Q}}=\spec \mathbb{Q}(\mu_{12})$. We can view 
$R_{/\mathbb{Q}}$ as subscheme of $E_{/\mathbb{Q}}$ in such a way that 
the points of $R_{/\mathbb{Q}}(\bar{\mathbb{Q}})$ correspond to the points of 
$R$ in $E(\bar{\mathbb{Q}})$. Therefore we denote 
$R_{/\mathbb{Q}}$ also by $R$. \\
Consider $t_1,t_2$ as rational functions on $E$. Clearly we have 
$\{t_1,t_2\}\in$ \linebreak
$H^2_{\cal M}(E,\mathbb{Q}(2))$.  
Since $t_i(Q)$ is a root of unity for $i\in \{1,2\}$ and for every 
$Q\in R$ we can also view 
$\{t_1,t_2\}$ as an element of $H^2_{\cal M}(E,R;\mathbb{Q}(2))$. 
We have 
$r_{\cal D}(\{f,g\})=[C_2]$ even in the relative situation. \\
In what follows set 
\begin{equation*}
\Gamma:=[C_2] \in 
H^2_{\cal D}((E,R)_{/\mathbb{R}},\mathbb{R}(2))=
H^1((E,R)/\mathbb{R},\mathbb{R}(1))
\end{equation*}
and view $\gamma:=[A]\otimes (2\pi i)^{-1}$ as an element of 
$H_1((E,R)/\mathbb{R};\mathbb{Z}(-1))$. With this notation 
(\ref{th_gives_us}) reads  
\begin{equation*}
m(P)=\pm \langle \Gamma,\gamma\rangle.   
\end{equation*} 
One has the following birational map on $Z(P)$:
\begin{align*}
\sigma:Z(P) & \rightarrow Z(P) \\
 (t_1,t_2) &\mapsto \left(t_1,\frac{t_2}{-1-2t_1t_2}\right).
\end{align*}
This can be extended to an involution on $E$, which is defined over 
$\mathbb{Q}$.  It is easy to see that $\sigma|_R=\id$ and that for 
$[C_2]\in H^2_{\cal D}(E_{/\mathbb{R}},\mathbb{R}(2))=
H^1(E/\mathbb{R},\mathbb{R}(1))$ 
\begin{equation*}
\sigma^*[C_2]=-[C_2].
\end{equation*}
From relative, long exact sequences of algebraic topology we get for cohomology
\begin{equation*}
\ldots \rightarrow H^0(R/\mathbb{R},\mathbb{R}(1)) 
\stackrel{\delta^*}{\rightarrow}
H^1((E,R)/\mathbb{R};\mathbb{R}(1))\stackrel{j^*}{\rightarrow} 
H^1(E/\mathbb{R},\mathbb{R}(1)) \rightarrow 0.
\end{equation*}
and for homology
\begin{equation*}
0\rightarrow H_1(E/\mathbb{R},\mathbb{R}(-1))\stackrel{j_*}{\rightarrow}
H_1((E,R)/\mathbb{R};\mathbb{R}(-1)) \stackrel{\delta_*}{\rightarrow} 
H_0(R/\mathbb{R},\mathbb{R}(-1))\rightarrow \ldots.
\end{equation*}
Here $j:(E,\emptyset)\rightarrow (E,R)$ denotes 
the inclusion and $\delta_*$ (respectively $\delta^*$) the boundary 
operator. \\
Our involution $\sigma$ gives us an $\langle\sigma\rangle$-operation on all of
the above groups which makes all occuring homomorphisms and the several 
pairings $\langle .,.\rangle$ equivariant.
Forcing our sequences to be short exact and choosing 
$\langle\sigma\rangle$-equivariant homomorphisms $s$ and $t$ 
subject to the conditions  
$j^*\circ s=\id$ and $\delta_*\circ t=\id$ we have the two decompositions
\begin{align*}
\Gamma &= \delta^*\Gamma^0 + s\Gamma^1 && \text{and} \\
\gamma &= j_*\gamma_1 + t\gamma_0
\end{align*}
where
\begin{align*}
\Gamma^0 &\in H^0(R/\mathbb{R},\mathbb{R}(1)), \\
\Gamma^1 &\in H^1(E/\mathbb{R},\mathbb{R}(1)), \\
\gamma_0 &\in H_0(R/\mathbb{R},\mathbb{Q}(-1)) && \text{ and} \\
\gamma_1 &\in H_1(E/\mathbb{R},\mathbb{Q}(-1)).
\end{align*}
Clearly we have 
$\Gamma^1=[C_2]\in H^1(E/\mathbb{R},\mathbb{R}(1))$ and therefore
$\sigma \Gamma^1=-\Gamma^1$. Using this and the fact that $\sigma$ operates 
trivially on $H_0(R/\mathbb{R},\mathbb{Q}(-1))$ we get 
\begin{equation*}
m(P)= 
\pm\underbrace{\langle \Gamma^1,\gamma_1\rangle}_{I}
\pm \underbrace{\langle\Gamma^0,\gamma_0\rangle}_{II}. 
\end{equation*} 
Our usual 
reasoning shows us (modulo Beilinson's conjectures) that 
\begin{equation*}
\langle \Gamma^1,\gamma_1\rangle = * L'(E,0)
\end{equation*} 
which takes care of term I. To give a proper meaning to term II is
much harder. It turns out that we are not totally free in choosing a
splitting of the above cohomology sequence. \\
For every two geometric points of $R$ the difference is a
torsion point of the elliptic curve $E$. It should be exact this property 
(as we will indicate in the next chapter) that allows us to choose 
splittings $s$ which make the whole following diagram commute:
\begin{equation}  \label{diagram_commute}
\begin{array}{ccccccc}
H_{\cal M}^1(R,\mathbb{Q}(2)) & \stackrel{\delta^*}{\rightarrow} & 
H_{\cal M}^2(E,R;\mathbb{Q}(2)) 
& \genfrac{}{}{0pt}{}{\stackrel{s}{\leftarrow}}{\stackrel{j^*}{\rightarrow}} 
& H^2_{\cal M}(E,\mathbb{Q}(2)) & \rightarrow & 0 \\
\Big\downarrow\vcenter{\rlap{$r_{\cal D}$}} & 
 & \Big\downarrow\vcenter{\rlap{$r_{\cal D}$}} & 
 & \Big\downarrow\vcenter{\rlap{$r_{\cal D}$}} & &    \\ 
H_{\cal D}^1(R_{/\mathbb{R}},\mathbb{R}(2)) 
& \stackrel{\delta^*}{\rightarrow} & 
H_{\cal D}^2((E,R)_{/\mathbb{R}};\mathbb{R}(2)) 
& \genfrac{}{}{0pt}{}{\stackrel{s}{\leftarrow}}{\stackrel{j^*}{\rightarrow}}  
& H^2_{\cal D}(E_{/\mathbb{R}},\mathbb{R}(2)) & 
\rightarrow & 0 \\
\Big\| & & \Big\| & & \Big\| & &    \\ 
 H^0(R/\mathbb{R},\mathbb{R}(1)) & 
\stackrel{\delta^*}{\rightarrow} & H^1((E,R)/\mathbb{R};\mathbb{R}(1)) & 
\genfrac{}{}{0pt}{}{\stackrel{s}{\leftarrow}}{\stackrel{j^*}{\rightarrow}} 
& H^1(E/\mathbb{R},\mathbb{R}(1)) &\rightarrow & 0.
\end{array}
\end{equation}
The element $B:=\{f,g\}\in H_{\cal M}^2(E,R;\mathbb{Q}(2))$ decomposes 
as follows 
\begin{equation*}
B =\delta^*B^0+sB^1  
\end{equation*} 
where 
\begin{align*}
B^0 &\in H_{\cal M}^1(R,\mathbb{Q}(2)) && \text{and} \\ 
B^1 &\in H_{\cal M}^2(E,\mathbb{Q}(2)).
\end{align*}
Evaluating $r_{\cal D}$ at $B$ and comparing to $\Gamma$ gives 
\begin{equation*}
\delta^*\Gamma^0+s\Gamma^1=(r_{\cal D}\circ\delta^*)B^0+
(r_{\cal D}\circ s) B^1. 
\end{equation*} 
Using $r_{\cal D}\circ s=s\circ r_{\cal D}$ and 
$r_{\cal D}B^1=\Gamma^1$ we get  
\begin{equation*}
r_{\cal D}B^0-\Gamma^0\in \ker\delta^*
\end{equation*}
which in turn amounts to
\begin{equation*} 
\langle\Gamma^0,\gamma_0\rangle=\langle r_{\cal D}B^0,\gamma_0\rangle. 
\end{equation*}
Using Beilinson's theorem on special values of Dirichlet L-functions
and especially his explicit description of the regulator map 
\begin{equation*}
r_{\cal D}: H^1_{\cal M}(R,\mathbb{Q}(2)) \rightarrow 
H^1_{\cal D}(R_{/\mathbb{R}},\mathbb{R}(2)) 
\end{equation*}
(see \cite{beil85} or \cite{neuk88}) we can interprete 
$\langle r_{\cal D}B^0,\gamma_0\rangle$ in the 
way we intended to do. 
In our case Beilinson's theorem says that there is a map 
\begin{equation*}
\varepsilon_2:\mu_{12}-\{1\}\rightarrow H^1_{\cal M}(R,\mathbb{Q}(2)),
\end{equation*}
such that 
\begin{equation*}
H^1_{\cal M}(R,\mathbb{Q}(2))=
\mathbb{Q}\cdot(\varepsilon_2(\xi)-\varepsilon_2(\xi^{-1}))\oplus 
\mathbb{Q}\cdot(\varepsilon_2(\xi^5)-\varepsilon_2(\xi^{-5})). 
\end{equation*}
for $\xi=e^{\pi i/6}$. Now let $\psi$ be the non-trivial Dirichlet character 
of $\mathbb{Z}/4\mathbb{Z}$ and let $\chi$ be as above. Beilinson's theorem 
further tells us
\begin{align*}
r_{\cal D}(\varepsilon_2(\xi)-\varepsilon_2(\xi^{-1}))& =
(q_1 L'(\psi,-1)+q_2 L'(\chi,-1)) \eta \\
 & \quad + (q_1 L'(\psi,-1)-q_2 L'(\chi,-1)) \eta'
\end{align*}
and 
\begin{align*}
r_{\cal D}(\varepsilon_2(\xi^5)-\varepsilon_2(\xi^{-5}))&=
(q_1 L'(\psi,-1)-q_2 L'(\chi,-1)) \eta \\
 & \quad + (q_1 L'(\psi,-1)+q_2 L'(\chi,-1)) \eta'
\end{align*}
for $q_1,q_2\in \mathbb{Q}^*$. Here
\begin{equation*} 
\begin{array}{ccc}
\eta=\left(
\begin{array}{c}
2\pi i \\ 0 \\ -2\pi i \\ 0
\end{array}
\right) & \text{and} &
\eta'=\left(
\begin{array}{c}
0 \\ 2\pi i \\ 0 \\ -2\pi i 
\end{array}
\right)
\end{array}
\end{equation*} 
is a $\mathbb{Q}$-basis of $H^1_{\cal D}(R_{/\mathbb{R}},\mathbb{R}(2))=
[(\mathbb{C}/\mathbb{R}(2))^{R(\mathbb{C})}]^+$.
Therefore it shows up
that we have 
\begin{equation} 
\langle r_{\cal D}B^0,\gamma_0\rangle= 
\kappa_1 L'(\chi,-1)+\kappa_2 L'(\psi,-1)
\end{equation} 
where $\kappa_1,\kappa_2\in \mathbb{Q}$. Together with term I we get   
\begin{equation*} 
m(P)=* L'(E,0)+\kappa_1 L'(\chi,-1)+\kappa_2 L'(\psi,-1).
\end{equation*}
Unfortunately this falls short of ``proving'' 
(\ref{evidence_suggests}).

\section{A general philosophy}

\label{five}
In this last chapter let us first have a look at another interesting
example due to Boyd. Let 
\begin{equation} \label{example_due_to_Boyd_Let}
P(t_1,t_2):=t_1^2 t_2^2+t_1+t_2+1.
\end{equation} 
The zero locus $Z(P)$ is birationally equivalent to the elliptic curve
$E$ defined by 
\begin{equation*}
y^2+y=x^3-x^2.
\end{equation*}
An easy calculation gives 
\begin{equation*} 
Z(P)\cap T^2=\{(-\zeta,\zeta)|\zeta^4=-1\}\cup\{(-1,-1)\}.
\end{equation*}
Denote this set by $R$ and consider it as a subvariety of $E$.    
Boyd has calculated $m(P)$ numerically using a precision of $25$
decimal places. 
He gave this value together with the numerical values of $L'(E,0)$ and
$L'(\chi,-1)$ for some Dirichlet characters of conductor $8$ as an
input to a
linear dependence finder (like for example {\it lindep} in
the package
{\it Pari}; see \cite{bat_et_al96}). An intensive search using this
method failed to produce a formula like (\ref{such_as}) or 
(\ref{evidence_suggests}).  \\
Applying \ref{following} to the polynomial 
$P(t_1t_2,t_2)=t_1^2t_2^4+t_1t_2+t_2+1$ we easily 
see that  
\begin{equation*}
m(P)=\pm\langle r_{\cal D}(\Psi),\Phi\rangle
\end{equation*} 
for certain elements $\Psi\in H^2_{\cal M}(E,R;\mathbb{Q}(2))$ and  
$\Phi\in H_1(E,R;\mathbb{Q}(-1))$. 
The question that arises from these considerations is: what is the
basic difference between examples like (\ref{Boyd_is_given_by}) and 
(\ref{example_due_to_Boyd_Let})? \\
Denote by $P_1,\ldots,P_5$ the five geometric points of $R$ (again
considered as points on $E$). 
Using {\it Pari} the author has calculated the multiples 
$[m](P_i-P_j)$ on $E$ for $i\neq j$ and $m\le 1000$. These
calculations give strong evidence that the $P_i-P_j$ for $i\neq j$ 
are no torsion points on $E$ at all. \\
As is easily seen the geometric points of the boundary $R$ of 
chapter \ref{four} have this property of every difference of points being a
torsion point on the elliptic curve. This should answer our above
question since in what follows we will give a heuristical argument why
the mentioned property is crucial in finding a splitting in $K$-theory like 
(\ref{diagram_commute}). \\
Let $K/\mathbb{Q}$ be a finite Galois extension subject to the
condition that all geometric points of $R$ are $K$-rational. Denote by
${\cal MM}_K$ the (not yet constructed) category of mixed motives over
$K$. Set $U:=E-R$. By base extension and by applying the
functor $H^*$ we get motives $H^*(E_K),H^*(U_K)$ and $H^*(R_K)$. Take a
look at the exact sequence 
\begin{equation*}
0\rightarrow H^1(E_K)(1)\rightarrow H^1(U_K)(1)
\rightarrow  H^2(E_K,U_K)(1)
\stackrel{j^*}{\rightarrow} H^2(E_K)(1)
\end{equation*} 
in ${\cal MM}_K$ and force it to be a short exact sequence 
\begin{equation} \label{short_exact_sequence}
0\rightarrow H^1(E_K)(1)\rightarrow H^1(U_K)(1)
\rightarrow \ker(j^*)\rightarrow 0.
\end{equation}  
General motivic folklore states that the splitting of 
(\ref{short_exact_sequence}) in ${\cal MM}_K$ is equivalent to our above
condition. Let us assume that this condition holds. The sequence 
(\ref{short_exact_sequence}) is dual to   
\begin{equation} \label{is_dual_to} 
0\rightarrow \img(\delta) \rightarrow  H^1(E_K,R_K)
\rightarrow H^1(E_K)\rightarrow 0
\end{equation}   
where $\delta:H^0(R_K)\rightarrow  H^1(E_K,R_K)$. Again according to
general motivic folklore $\ext^1(\mathbb{Q}(0),.)$ groups in 
${\cal MM}_K$ resp. $\ext^1(\mathbb{R}(0),.)$ groups in 
a certain category of mixed Hodge structures ${\cal MH}_R$ are 
naturally isomorphic to motivic cohomology resp. Deligne-Beilinson 
cohomology. Therefore applying those functors to  
(\ref{is_dual_to}) we get a compatible 
splitting 
\begin{equation*}
\begin{array}{ccccc} 
H_{\cal M}^2(E_K,R_K;\mathbb{Q}(2)) 
& \genfrac{}{}{0pt}{}{\stackrel{s}{\leftarrow}}{\rightarrow} 
& H^2_{\cal M}(E_K,\mathbb{Q}(2)) & \rightarrow & 0 \\
\Big\downarrow\vcenter{\rlap{$r_{\cal D}$}} & 
 & \Big\downarrow\vcenter{\rlap{$r_{\cal D}$}} & &    \\  
H_{\cal D}^2(E_K,R_K;\mathbb{R}(2)) 
& \genfrac{}{}{0pt}{}{\stackrel{s}{\leftarrow}}{\rightarrow}  
& H^2_{\cal D}(E_K,\mathbb{R}(2)) & 
\rightarrow & 0. 
\end{array}
\end{equation*}
Using Galois descent we finally produce a splitting like 
(\ref{diagram_commute}). \\ 

\vspace{1cm}
\hspace{6cm} 
{\small
\begin{tabular}{ll}
Hubert Bornhorn & \\
Westf{\"a}lische Wilhelms-Universit{\"a}t M{\"u}nster & \\
Mathematisches Institut & \\
Einsteinstra{\ss}e 64 & \\
48149 M{\"u}nster & \\ 
Germany & \\
Hubert.Bornhorn@t-online.de & 
\end{tabular}
}



\bibliographystyle{amsplain}
\bibliography{mkl}

\end{document}